\documentclass{article}
\usepackage[utf8]{inputenc}
\usepackage{amsthm}
\usepackage{amsmath}
\usepackage{comment}
\usepackage{amsfonts}
\usepackage{color}
\usepackage{latexsym}
\usepackage{geometry}
\usepackage{enumerate}
\usepackage[shortlabels]{enumitem}
\usepackage{csquotes}
\usepackage{graphicx}
\usepackage{comment}
\usepackage{appendix}
\usepackage{hyperref}
\usepackage{esint}
\usepackage{bm}
\usepackage{amssymb}

\hypersetup{
    colorlinks=true,
    linkcolor=blue,
    filecolor=magenta,      
    urlcolor=cyan,
}

\emergencystretch=\maxdimen
\DeclareMathOperator{\Rm}{Rm}
\DeclareMathOperator{\Ric}{Ric}

\makeatletter
\newcommand*{\rom}[1]{\rm {\expandafter\@slowromancap\romannumeral #1@}}
\makeatother
\def\Xint#1{\mathchoice
{\XXint\displaystyle\textstyle{#1}}%
{\XXint\textstyle\scriptstyle{#1}}%
{\XXint\scriptstyle\scriptscriptstyle{#1}}%
{\XXint\scriptscriptstyle\scriptscriptstyle{#1}}%
\!\int}
\def\XXint#1#2#3{{\setbox0=\hbox{$#1{#2#3}{\int}$ }
\vcenter{\hbox{$#2#3$ }}\kern-.6\wd0}}

\def\dashint{\Xint-}

\protected\def\vts{%
  \ifmmode
    \mskip0.5\thinmuskip
  \else
    \ifhmode
      \kern0.08334em
    \fi
  \fi
}

\numberwithin{equation}{section}
\newtheorem{Theorem}{Theorem}[section]

\newtheorem{Lemma}[Theorem]{Lemma}

\theoremstyle{definition}
\newtheorem{Definition}[Theorem]{Definition}

\title{Rigidity and gap theorems for Ricci shrinkers}
\author{Pak-Yeung Chan\footnote{Pak-Yeung Chan's research is supported by the Yushan Young Fellow Program of the Ministry of Education (MOE) Taiwan (MOE-108-YSFMS-0004-012-P1), and by the NSTC grant 113-2115-M-007-014-MY2.}, Yongjia Zhang\footnote{Yongjia Zhang's research is  supported by National Natural Science Foundation of China NSFC12301076.}}
\date{}

\begin{document}
\maketitle
\begin{abstract}
We prove local versions of the Ricci curvature and $\nu$-entropy gap theorems for Ricci shrinkers, which respectively generalize a previous result of Munteanu-Wang \cite{MW11} and a prior result of the authors with Ma \cite{CMZ22}. The key point is that these local gaps depend only on the dimension and not on the global entropy or any other geometric information of the Ricci shrinker. As an application, we provide a local criterion for removable Type~I singularities of the Ricci flow.
\end{abstract}
\section{Introduction}

A Ricci shrinker is a tuple $(M^n,g,f)$ of a complete Riemannian manifold $(M,g)$ and a smooth potential function $f:M\to\mathbb{R}$, satisfying the equation
\begin{align}\label{eq:shrinker_equation}
    \Ric+\nabla^2 f=\frac{1}{2}g.
\end{align}
In this article, we always normalize the potential function such that the following equation holds:
\begin{align}\label{eq:normalization}
    |\nabla f|^2+R=f.
\end{align}
By \cite{CZ10}, the potential function $f$ always admits a minimum point $p$. Since this point plays a central role in the study of Ricci shrinkers, we will use the tuple $(M,g,f,p)$ to represent a shrinker. 

The most trivial example of shrinkers is the gaussian shrinker $(\mathbb{R}^n,g_e,|x|^2/4,0)$, where $g_e$ is the standard flat metric on the Euclidean space. The rigidity of Gaussian shrinkers has attracted considerable attention in the literature. For instance, Carrillo-Ni \cite{CN09} and Yokota \cite{Yo12} proved a gap theorem for the gaussian shrinker with respect to Perelman's $\mu$-functional \cite{Per02}. This result was later localized by the authors and Ma \cite{CMZ22}. On the other hand, Li-Wang \cite{LW20}  studied the rigidity of the gaussian shrinker with respect to the Gromov-Hausdorff distance. They proved that a normalized shrinker, if close enough to the gaussian shrinker in the Gromov-Hausdorff sense, must be isometric to the gaussian shrinker. For other rigidity results of shrinkers, especially of the cylindrical shrinkers, see  \cite{LW24,CM25,LZ23, WW25, Zh18, Zh20}.

In this article, we continue studying the rigidity of the gaussian shrinker. We shall find conditions as local and weak as possible that are sufficient to force a shrinker to be flat. Our work is inspired by \cite{LLW21}. Indeed, though not explicitly stated, the arguments of \cite{LLW21} imply the following strong rigidity result 
of the gaussian shrinker.

\begin{Theorem}[\cite{LLW21}]\label{Thm:GH gap}
There is a positive constant $\varepsilon=\varepsilon(n)>0$ with the following property. Let $(M^n,g,f,p)$ be a Ricci shrinker, where $p$ is a minimum point of $f$. Suppose 
\begin{align*}
    d_{\operatorname{GH}}\Big(B(p,1),B_e(1)\Big)<\varepsilon,
\end{align*}
where $d_{\operatorname{GH}}$ is the Gromov-Hausdorff distance, and $B_e(1)$ is the unit ball in the Euclidean space,
    then $(M^n,g,f,p)$ is the gaussian shrinker $(\mathbb{R}^n,g_e,|x|^2/4,0)$.
\end{Theorem}


It is interesting to ask, whether we can find other local conditions that are sufficient to force a shrinker to be gaussian. Naber \cite[Theorem 8.1]{Na10} proved that a complete Ricci shrinker must be gaussian if $\frac{1}{2} > \Ric \ge 0$ everywhere. Munteanu-Wang \cite[Corollary 1.1]{MW11} showed that a shrinker is flat if $|{\Ric}| \le \frac{1}{100n}$ globally on $M$. The authors and Ma \cite{CMZ22} proved that a shrinker with sufficiently large curvature radius (depending only on the dimension) at a minimum point of $f$ must be flat. In light of these results, as well as those of Li-Li-Wang \cite{LLW21}, it is natural to ask whether the same conclusion holds under a small local Ricci curvature bound near a minimum point of $f$. We show that this is indeed the case.

\begin{Theorem} \label{Thm:Ric gap}
There is a positive constant $\varepsilon=\varepsilon(n)>0$ with the following property. Let $(M^n,g,f,p)$ be a Ricci shrinker, where $p$ is a minimum point of $f$. Suppose 
\begin{align*}
    \Ric\le \varepsilon g\quad \text{ on }\quad B(p,1),
\end{align*}
    then $(M^n,g,f,p)$ is the gaussian shrinker $(\mathbb{R}^n,g_e,|x|^2/4,0)$.
\end{Theorem}

\bigskip

\noindent\emph{Remarks.}
\begin{enumerate}[(1)]
    \item Since on a shrinker, we always have $R>0$ by a result of Chen \cite{Che09}, it makes no difference to assume $|{\Ric}|\le \varepsilon$ or $\Ric\le \varepsilon g$; we will freely implement the former assumption in our argument.
    \item According to our argument, the theorem can be improved straightforwardly to the following: For any $r>0$, there is an $\varepsilon=\varepsilon(r,n)>0$, such that if
    \begin{align*}
        \Ric\le \varepsilon g\quad \text{ on }\quad B(p,r),
    \end{align*}
    then the shrinker is gaussian. 
    \item Li-Li-Wang \cite[Theorem 1.3]{LLW21} obtained a local gap theorem for the scalar curvature of closed shrinkers, where the gap also depends on the entropy. Wang-Wang \cite{WW25} established a rigidity theorem for shrinkers in terms of the volume ratio on a large ball centered at $p$.
\end{enumerate}

In \cite{CMZ22}, it is also proved that a shrinker with sufficiently small local $\mu$-functional in a sufficiently large ball at the minimum point of the potential function must be flat. We will also provide an improvement of this theorem.

\begin{Theorem} \label{Thm:nu gap}
There is a positive constant $\varepsilon=\varepsilon(n)>0$ with the following property. Let $(M^n,g,f,p)$ be a Ricci shrinker, where $p$ is the minimum point. Suppose 
\begin{align*}
    \nu\Big(B(p,1),g,1\Big)\ge-\varepsilon,
\end{align*}
    then $(M^n,g,f,p)$ is the gaussian shrinker $(\mathbb{R}^n,g_e,|x|^2/4,0)$.
\end{Theorem}

Since, by \cite{EMT11}, smooth Ricci shrinkers are singularity models (fixed-point blow-up limits) of Type~I Ricci flows, and since a flat blow-up limit implies that the base point is not a singular point \cite[Theorem 1.1]{EMT11}, we may apply Theorem \ref{Thm:Ric gap} to obtain the following result concerning removable singularities for Type~I Ricci flows.

\begin{Theorem}[A local condition for removable Type I singularities]\label{Thm:Applications}
    There is a dimensional constant $\delta=\delta(n)>0$ with the following property. Let $(M^n,g_t)_{t\in[-T,0)}$ be a Type I  Ricci flow (that is, $|{\Rm_{g_t}}|\le C/|t|$ for some $C$) on a close manifold $M^n$. Assume $x\in M$ is a point such that  $r_c(x,t)\ge 6n\sqrt{|t|}$ for all $t\in  [-T,0)$, where 
        \begin{align*}
            r_c(x,t):=\sup\{r>0\,:\, \Ric_{g_t}\le \delta r^{-2}g_t\ \text{ on }\ B_{g_t}(x,r)\}.
        \end{align*}
    Then $|{\Rm_{g_t}}|(x)$ remains bounded as $t\to T$. In other words, $x$ is not a singular point according to the definition of \cite{EMT11}.
\end{Theorem}

\emph{Remarks.} 
\begin{enumerate}[(1)]
    \item Note that the constant $\delta=\delta(n)$ is a dimensional constant, and it is in particular independent of the Type-I-coefficient $C$. On the other hand, a global version of the theorem can be deduced from \cite{MW11}, namely, if $r_c(x,t)\ge \sqrt{|t|}$ for {\sl all} $(x,t)\in M\times[-T,0)$, then the flow does not admit any singularity at $t=0$. 
    \item One may also apply Theorem \ref{Thm:nu gap} to obtain a condition of removable singularities formulated with the local $\nu$-functional. However, this result is not different from that which one can deduce from Perelman's pseudolocality theorem \cite[Theorem 10.1]{Per02}.
\end{enumerate}

\textbf{Acknowledgement.} The second author would like to thank Professor Lei Ni, Professor Liang Cheng, and Professor Meng Zhu for some enlightening discussions. The authors would like to thank Professor Bennett Chow for providing some helpful unpublished manuscripts.

\section{Preliminaries}

In this section, we summarize the main techniques for the proof of our main theorems. 

\subsection{Harmonic radius estimate}

First of all, we introduce an $\varepsilon$-regularity due to Anderson \cite{And90}. Let us recall the definition of harminic radius.

\begin{Definition}[Harmonic radius]
    Let $(M^n,g)$ be a complete Riemannian manifold and $p\in M$. The $(Q,\alpha)$-harmonic radius $r_H(p)$ at $p$, where $Q>1$ and $\alpha\in(0,1)$, is defined to be the maximum radius $r$ such that: 
    \begin{enumerate} [(1)]
        \item there is a harmonic coordinate chart $x=(x^1,x^2,\hdots,x^n): B(p,r)\to \mathbb{R}^n$. In particular, each coordinate $x^i$ is a harmonic
        function;
        \item \emph{in this chart}, it holds that:
        \begin{itemize}
            \item $\displaystyle Q^{-1}\delta_{ij}\le g_{ij}\le Q g_{ij}$;
            \item $\displaystyle \sup_{x}r|\partial g_{ij}(x)|+\sup_{x\neq y}r^{1+\alpha}\frac{|\partial g_{ij}(x)-\partial g_{ij}(y)|}{|x-y|^\alpha}\le Q-1$.
        \end{itemize}
    \end{enumerate}
\end{Definition}
Henceforth throughout the whole paper, we shall \emph{fix the parameters $Q$ and $\alpha$} (for instance, we fix $\alpha =1/2$, $Q=2$), and will refer to $r_H(p)$ as the harmonic radius at $p$. Should it be ambiguous, we would also use the notation $r_H(p,g)$ to emphasize the metric under consideration. The following result is essentially the localization of \cite[Theorem 3.2, Remark 3.3]{And90}. For a proof of the local statement, see \cite[Theorem 2.2]{H19}.

\begin{Theorem}[$\varepsilon$-regularity]\label{Thm:critical radius}
    There is a positive number $\delta=\delta(n)>0$ with the following property. Let $(M^n,g)$ be a complete manifold and $p\in M$ a point. Suppose 
    \begin{align*}
        |B(p,1)|\ge (1-\delta)\omega_n\quad \text{ and }\quad |{\Ric}|\le \delta\quad\text{ on }\quad B(p,1),
    \end{align*}
    where $\omega_n$ is the volume of the unit ball in the $n$-dimensional Euclidean space, then we have
    \begin{align*}
        r_H(p)\ge \delta.
    \end{align*}
\end{Theorem}

\subsection{Conical function and approximate cone}

In the Cheeger-Colding theory \cite{CC97,CC00a,CC00b}, it is known that a tangent cone at any point of a noncollapsed Ricci limit space is a metric cone. This result can be more clearly formulated through the construction of a conical function. Let us recall the definition of conical functions.

\begin{Definition}\label{Def_conical function}
    Let $(M^n,g)$ be a complete Riemannian manifold and $p\in M$ a point. A function $h:B(p,r)\to\mathbb{R}$ is called an $(\varepsilon,r)$-conical function at $p$, if the following hold
    \begin{enumerate}[(1)]
        \item $\Delta h=2n$;
        \item $\displaystyle\dashint_{B(p,r)}|\nabla^2 h-2g|^2\,dg\le \varepsilon$;
        \item $\displaystyle\dashint_{B(p,r)}||\nabla h|^2-4h|\,dg\le \varepsilon r^2$;
        \item $\displaystyle\dashint_{B(p,r)}|\nabla (h-d^2(p,\cdot))|\,dg\le \varepsilon r$;
        \item $\displaystyle \sup_{B(p,r)}|h-d^2(p,\cdot)|\le \varepsilon r^2$,
    \end{enumerate}
    where $\dashint$ is the average integral.
\end{Definition}

If on a ball the Ricci curvature is almost nonnegative, and there exists a conical function, then the ball is close to a metric cone in the Gromov-Hausdorff sense. This is originally proved in \cite{CC96}; see also \cite{Ch01}. A clearer proof of the following theorem can be found in Jiang's notes \cite{J}.

\begin{Theorem}[\cite{CC96}, see also \cite{Ch01}]\label{Thm:conical function implying almost cone}
    Let $(M^n,g)$ be a complete manifold. Assume that there is an $(\varepsilon,r)$-conical function at some $p\in M$. Furthermore, assume $\Ric\ge -(n-1)\varepsilon r^{-2}$ on $B(p,r)$. Then there is a metric cone $(C,z^*)$, with $z^*$ being the tip, such that
    \begin{align*}
        d_{\operatorname{GH}}\Big(B(p,r/2),B(z^*,r/2)\Big)\le \Psi(\varepsilon)r.
    \end{align*}
\end{Theorem}

Note that in this paper, we adopt the notation $\Psi(a_1,a_2,\hdots\,|\,b_1,b_2,\hdots)$ to represent a positive number such that
\begin{align*}
    \Psi(a_1,a_2,\hdots\,|\,b_1,b_2,\hdots)\to 0 \qquad \text{ as }\qquad a_1\to0,a_2\to0,\hdots\ \text{ while }\ b_1,b_2,\hdots\ \text{ remain fixed.}
\end{align*}

\subsection{Structure of shrinker limit spaces}

In \cite{LLW21,HLW21}, Huang-Li-Li-Wang studied the ``Cheeger-Colding theory for Ricci shrinkers''and established several structural results for the Ricci shrinker limit spaces. We shall introduce a couple of results which we shall apply later.

Recall that for a shrinker $(M^n,g,f,p)$ normalized as \eqref{eq:shrinker_equation} and \eqref{eq:normalization}, the shrinker entropy $\mu(M)$ is defined to be the nonpositive number such that
\begin{align*}
    \int_M e^{f+\mu(M)}\,dg=(4\pi)^{n/2}.
\end{align*}
It is observed by Carrillo-Ni \cite{CN09} and Li-Wang\cite{LW20} that
\begin{align*}
    \mu(M)=\mu(g,1)=\nu(g), 
\end{align*}
where $\mu$ and $\nu$ are Perelman's functionals \cite{Per02}.

First of all, for a shrinker, volume noncollapsedness is equivalent to entropy noncollapsedness.

\begin{Lemma}[{\cite[Lemma 2.2, Lemma 2.3]{LLW21}}]\label{Lm:volume and entropy noncollapsed}
    Let $(M^n,g,f,p)$ be a normalized Ricci shrinker. Then for any $v>0$, there is a $Y<+\infty$, such that if $|B(p,1)|\ge v$, then $\mu(M)\ge -Y$. Furthermore, for any $Y<+\infty$, there is a $v>0$, such that if $\mu(M)\ge -Y$, then $|B(p,1)|\ge v$.
\end{Lemma}

Now we summarize the properties of noncollapsed shrinker limit spaces.

\begin{Theorem}[{\cite[Theorem 1.1, Proposition 9.9]{LLW21}}]\label{Thm:shrinker limit}
    Let $\{(M^n_i,g_i,f_i,p_i)\}_{i=1}^\infty$ be a sequence of normalized Ricci shrinkers. Suppose that there is a $Y<+\infty$, such that
    \begin{align*}
        \mu(M_i)\ge -Y\qquad \text{ for all }\qquad i.
    \end{align*}
    Then, after passing to a subsequence, we have
    \begin{align*}
        (M_i,d_{g_i},p_i)\xrightarrow{\ \ \text{ pointed-Gromov-Hausdorff}\ \ }(X,d,p),
    \end{align*}
    where $(X,d,p)$ is a metric space admitting a decomposition
    \begin{align*}
        X=\mathcal{R}\sqcup\mathcal{S},
    \end{align*}
    such that:
    \begin{enumerate}[(1)]
        \item The Hausdorff dimension of $\mathcal{S}$ is no greater than $n-4$;
        \item There is a shrinker structure on $\mathcal{R}$. In other words, $\mathcal{R}$ is a smooth manifold with a Riemannian metric $g$ and potential function $f$, such that 
        \begin{align*}
            \Ric +\nabla^2f=\frac{1}{2}g.
        \end{align*}
        \item $g_i$ and $f_i$  converge smoothly to $g$ and $f$ on $\mathcal{R}$, respectively;
        \item If $R_g$ vanishes at some point in $\mathcal{R}$, then $(X,d,p)$ is a Ricci flat cone at $p$.
    \end{enumerate}
\end{Theorem}

\section{The proof of Theorem \ref{Thm:GH gap}
}

Though not stated explicitly in \cite{LLW21}, the essential ideas of Theorem \ref{Thm:GH gap} 
are all contained therein. In this section, we briefly explain how this result 
can be induced from  \cite{LLW21}.

For Theorem \ref{Thm:GH gap}, arguing by contradiction,  we suppose that there is a sequence of Ricci shrinkers $\{(M_i^n, g_i,f_i, p_i)\}_{i=1}^\infty$ with 
\begin{align}\label{eq:contradic GH}
    d_{\operatorname{GH}}\Big(B_{g_i}(p_i,1),B_e(1)\Big)<\varepsilon_i,
\end{align}
where $\varepsilon_i\to 0$, but none of them is the gaussian shrinker.

Suppose for the moment that the sequence is noncollapsed, namely,  there is some $Y<+\infty$, such that for all $i$
\begin{align*}
    \mu(M_i)\ge -Y.
\end{align*}
Then, by Theorem \ref{Thm:shrinker limit}, we can take a limit for a subsequence. Denote by $(X,d,p)$ the limit space and $X=\mathcal{R}\sqcup\mathcal{S}$ its decomposition. Since $B(p,1)$ is isometric to $B_e(1)$, we immediately have that, $B(p,1)\subset\mathcal{R}$ and the curvature vanishes therein. Consequently, $(X,p)$ is a cone at $p$ by Theorem \ref{Thm:shrinker limit} (4). While the cone tip $p$ is a smooth point, such a cone can be nothing but the Euclidean space. Thus $(M_i,d_{g_i},p_i)$ converges to the Euclidean space in the pointed Gromov-Hausdorff sense. On the other hand, \cite[Corollary 7]{LW20} shows that any shrinker close enough to the Euclidean space in the pointed-Gromov-Hausdorff sense must be the gaussian shrinker. So $(M_i^n, g_i,f_i, p_i)$ is a gaussian shrinker for all $i$ large enough. This contradicts our assumption on the sequence.

Thus, it remains only to show that the sequence $\{(M_i^n, g_i,f_i, p_i)\}_{i=1}^\infty$ is noncollapsed. Indeed, this is nothing but \cite[Proposition 5.8]{LLW21}. However, since that theorem is stated in terms of the global Gromov-Hausdorff convergence, while we only assume a local condition, we shall sketch a proof of the noncollapsedness of the contradictory sequence.

For each $g_i$, we define the conformal metric
\begin{align*}
    \bar g_i:=e^{-\frac{2(f_i-f_i(p_i))}{n-2}}g_i.
\end{align*}
Then, by \cite[Theorem 3.6]{LLW21}, there is a dimensional constant $D$, such that
\begin{enumerate}[(1)]
    \item For each $\rho\in(0,D^{-1}]$, we have
    \begin{align*}
        d_{\operatorname{GH}}\Big(B_{g_i}(p_i,\rho),B_{\bar g_i}(p_i,\rho)\Big)\le D\rho^2.
    \end{align*}
    \item $\displaystyle |{\Ric_{\bar g_i}}|\le D^2$ on $B_{\bar g_i}(p_i,(10D)^{-1})$.
\end{enumerate}

Since (1) and \eqref{eq:contradic GH} imply that
\begin{align*}
    d_{\operatorname{GH}}\Big(B_{\bar g_i}(p_i,\rho),B_e(\rho)\Big)\le (D\rho+\Psi(\varepsilon_i\,|\,\rho))\rho,
\end{align*}
the Ricci curvature bound (2) together with Colding's volume continuity theorem \cite{Cold97}  shows that, if we fix $\rho=\rho(n)\ll 1$ to be small enough, then
\begin{align*}
    |B_{\bar g_i}(p_i,\rho)|\ge \tfrac{1}{2}\omega_n\rho^n=c(n)\qquad \text{ for all $i$ large enough.}
\end{align*}
Hence, there is a unform lower bound for $|B_{g_i}(p_i,1)|$, and we have proved the noncollapsedness for the sequence by Lemma \ref{Lm:volume and entropy noncollapsed}.

\section{The proof of Theorem \ref{Thm:Ric gap}}

The main techniques of the proof are the Cheeger-Colding theory and the structure theorem of the shrinker limit spaces \cite{LLW21}. First of all, we estimate a ``critical radius'' for shrinkers with small Ricci curvature in the unit ball. This radius is called the ``volume radius'' by \cite[Definition 4.1]{LLW21}.

\begin{Lemma}\label{Lm:critical radius estimate}
    There are dimensional constants $\varepsilon(n)>0$ and $r_0(n)>0$ such that if $(M,g,f,p)$ is a complete gradient shrinker with $\Ric\leq \varepsilon g$ on $B(p,1)$, then  
    \[
      r:=\sup\{s\in (0,1]: |B_g(p,s)|\ge \omega_n(1-\delta)s^n \}\ge r_0(n)
    \]
    Here $\delta=\delta(n)$ is the constant in the statement of Theorem \ref{Thm:critical radius}.
   \end{Lemma}
\begin{proof}
    We argue by contradiction. Suppose there exist sequences $\varepsilon_i\to 0$ and $(M_i,g_i, f_i, p_i)$ such that the corresponding 
    $$r_i=\sup\{s\in(0,1]\,:\,|B_{g_i}(p_i,s)|\ge \omega_n(1-\delta)s^n\}\to 0.$$
    Let $\bar g_i:=r_i^{-2}g_i$. Then 
    \begin{align}\label{eq:condition of blow-up sequence}
        |B_{\bar g_i}(p_i,1)|_{\bar g_i} = \omega_n(1-\delta),\qquad |{\Ric}(\bar g_i)|\leq \varepsilon_ir^2_i\quad \text{ on }\quad  B_{\bar g_i}(p_i,r_i^{-1}).
    \end{align}
    It immediately follows from Theorem \ref{Thm:critical radius} that
    \begin{align}\label{eq:harrad lower bound}
        r_H(p_i,\bar g_i)\ge \delta\quad \text{ for all $i$ large enough}.
    \end{align}
 
 Next, by the soliton equation, we have
    \[
    r^{-2}_i\Ric(\bar g_i)+\nabla^2_{\bar g_i} (r_i^{-2}f_i)=\frac 12 \bar g_i
    \]
    and the Ricci curvature bounds imply that \begin{align}\label{eq: f-Hess}
        \left(\tfrac12-\varepsilon_i\right)\bar g_i\le \nabla^2_{\bar g_i} (r_i^{-2}f_i)\le\left(\tfrac12+\varepsilon_i\right)\bar g_i \quad \text{ on }\quad B_{\bar g_i}(p_i,r_i^{-1})
    \end{align}
    for all $i$. Since $p_i$ is the minimum point, it is easy to see that $\bar f_i:=r_i^{-2}f_i$ is almost an conical function. Indeed, the reader can readily check that $\bar f_i:=r_i^{-2}f_i$ satisfies Definition \ref{Def_conical function} (2)---(4) up to a constant multiple. We shall construct conical functions starting with $\bar f_i$.
    
    For notational simplicity, we shall define
    \begin{align*}
        h_i(x)=4(\bar f_i(x)-\bar f_i(p_i))= 4 r_i^{-2}(f_i(x)-f_i(p_i)).
    \end{align*}
Let us fix an arbitrary $A\gg 1$. For $i$ large enough such that $r_i^{-1} > 2A$, we define $\bar h_i: B_{\bar g_i}(p_i,2A)\to \mathbb{R}$ to be the solution of the Dirichlet problem
\begin{eqnarray}\label{eq:Dirichlet}
    \Delta_{\bar g_i} \bar h_i\ =\ 2n &\text{  on  }& B_{\bar g_i}(p,2A)\\\nonumber
     \bar h_i\ =\ \,  h_i &\text{  on  }& \partial B_{\bar g_i}(p,2A).
\end{eqnarray}
Obviously, $\bar h_i$ satisfies (1) of Definition \ref{Def_conical function}. We shall verify conditions (2)---(5) of Definition \ref{Def_conical function} to show that $\bar h_i$ is a $(\Psi(\varepsilon_i\,|\, A),A)$-conical function at $p_i$.
\\

\noindent\textbf{Verification of (5) $\displaystyle \sup_{B_{\bar g_i}(p_i,A)}|\bar h_i-d_{\bar g_i}^2(p_i,\cdot) |\le \Psi(\varepsilon_i\,|\, A)$.}
\\

By the definition of $h_i$ and by \eqref{eq: f-Hess}, we have
\begin{align}\label{eq:h-Hess}
    2(1-2\varepsilon_i)\bar g_i\leq \nabla^2_{\bar g_i} h_i\leq 2(1+2\varepsilon_i)\bar g_i.
\end{align}
Since $p_i$ is the minimum point of $h_i$, and since $h_i(p_i)=0$, picking any unit geodesic $\gamma(t)$ with $\gamma(0)=p_i$, and integrating over
\begin{align*}
  \frac{d^2}{dt^2}h_i(\gamma(t)) =\nabla^2 h_i(\gamma'(t),\gamma'(t))=2(1\pm2\varepsilon_i) 
  \\
  \left.\frac{d}{dt}\right|_{t=0}h_i(\gamma(t))=\langle\nabla h_i(p_i),\gamma'(0)\rangle=0,
\end{align*}
we obtain that, if $t\le 2A$, then
\begin{align}\label{eq:grad-lower}
    |\nabla_{\bar g_i}h_i|(\gamma(t))\ge \langle\nabla _{\bar g_i}h_i,\gamma'(t)\rangle\ge 2t-4\varepsilon_iA,
    \\\nonumber
    h_i(\gamma(t))=t^2\pm8\varepsilon_iA^2.
\end{align}
In other words, we have
\begin{align}\label{eq:h-close-to-d2}
    \sup_{B_{\bar g_i}(p_i,2A)}|h_i-d_{\bar g_i}^2(p,\cdot)|\le 8\varepsilon_iA^2.
\end{align}

To get the same bound for $\bar h_i$, let us observe that by \eqref{eq:h-Hess} 
\begin{align*}
    2n(1-2\varepsilon_i)\le \Delta_{\bar g_i}h_i\le 2n(1+2\varepsilon_i).
\end{align*}
In view of the definition \eqref{eq:Dirichlet} of $\bar h_i$, we consider the auxiliary function
\begin{align*}
    \varphi_i:=\bar h_i-\frac{1}{1+2\varepsilon_i}h_i-\frac{16\varepsilon_iA^2}{1+2\varepsilon_i}.
\end{align*}
On $\partial B_{\bar g_i}(p_i,2A)$, we have
\begin{align*}
    \varphi_i=\frac{2\varepsilon_i}{1+2\varepsilon_i}h_i-\frac{16\varepsilon_iA^2}{1+2\varepsilon_i}\le \frac{2\varepsilon_i(4A^2+8\varepsilon_iA^2)}{1+2\varepsilon_i}-\frac{16\varepsilon_iA^2}{1+2\varepsilon_i}<0.
\end{align*}
Note that we have applied \eqref{eq:h-close-to-d2}. In the interior of $B_{\bar g_i}(p_i,2A)$, we have
\begin{align*}
    \Delta_{\bar g_i}\varphi_i =\Delta_{\bar g_i}\bar h_i-\frac{1}{1+2\varepsilon_i}\Delta_{\bar g_i}h_i \ge 2n-\frac{(1+2\varepsilon_i)2n}{1+2\varepsilon_i}=0.
\end{align*}
The standard maximum principle immediately implies that $\varphi_i\le 0$ on $B_{\bar g_i}(p_i,2A)$. By a similar argument for the auxiliary function
\begin{align*}
    \psi_i:=\bar h_i-\frac{1}{1-2\varepsilon_i}h_i+\frac{16\varepsilon_i A^2}{1-2\varepsilon_i},
\end{align*}
we may obtain $\psi_i\ge 0$ on $B_{\bar g_i}(p_i,2A)$. In summary, we get
\begin{align}\label{eq:h bar h closeness}
    \sup_{B_{\bar g_i}(p_i,2A)}|\bar h_i-h_i|\le \Psi(\varepsilon_i\,|\,A).
\end{align}
Taking \eqref{eq:h-close-to-d2} into account, we have verified 
\begin{align}\label{eq:item 5}
    \sup_{B_{\bar g_i}(p_i,2A)}|\bar h_i-d_{\bar g_i}^2(p,\cdot)|\le \Psi(\varepsilon_i\,|\,A),
\end{align}
that is, (5) of Definition \ref{Def_conical function}.
\\

\noindent\textbf{Verification of (3) $\displaystyle \dashint_{B_{\bar g_i}(p_i,A)}\big||\nabla_{\bar g_i}\bar h_i|^2-4 \bar h_i\big|\, d\bar g_i\leq \psi(\varepsilon_i \,|\,A) A^2$.}
\\

We still begin our argument with $h_i$. Take any unit geodesic $\gamma(t)$ with $\gamma(0)=p_i$, we have
\begin{align*}
    \frac{d}{dt}|\nabla_{\bar g_i}h_i|^2(\gamma(t))=&\ 2\nabla_{\bar g_i}^2h_i(\nabla_{\bar g_i} h_i,\gamma'(t))
    \\
    \le &\ 4(1+2\varepsilon_i)\langle\nabla_{\bar g_i} h_i,\gamma'(t)\rangle \le 4(1+2\varepsilon_i)|\nabla_{\bar g_i} h_i|(\gamma(t)),
\end{align*}
where we have applied \eqref{eq:h-Hess}. Thus,
\begin{align*}
    \frac{d}{dt}|\nabla_{\bar g_i}h_i|(\gamma(t)) \le 2(1+2\varepsilon_i).
\end{align*}
Integrating this inequality, we get, for all $t\le 2A$, 
\begin{align}\label{eq:grad h upper bound}
    |\nabla_{\bar g_i}h_i|(\gamma(t)) \le |\nabla_{\bar g_i}h_i|(\gamma(0))+2(1+2\varepsilon_i)t \le 2t+4\varepsilon_iA.
\end{align}
Combining the above with \eqref{eq:grad-lower} and \eqref{eq:h-close-to-d2}, we get
\begin{align}\label{eq:h radial}
    \sup_{B_{\bar g_i}(p_i,2A)}\big||\nabla_{\bar g_i}h_i|^2-4h_i\big|\le \Psi(\varepsilon_i\,|\, A).
\end{align}

Next, we show that the gradients of $h_i$ and $\bar h_i$ are $L^2$-close. Indeed, 
\begin{align}\label{eq:h bar h grad L2 close}
    \dashint_{B_{\bar g_i}(p_i,2A)}|\nabla_{\bar g_i}(h_i-\bar h_i)|^2\, d\bar g_i= &\ - \dashint_{B_{\bar g_i}(p_i,2A)}(h_i-\bar h_i)\Delta(h_i-\bar h_i)\, d\bar g_i
    \\\nonumber
    \le &\ \sup_{B_{\bar g_i}(p_i,2A)}|h_i-\bar h_i|\cdot |\Delta(h_i-\bar h_i)|
    \\\nonumber
    \le &\ \Psi(\varepsilon_i\,|\,A),
\end{align}
where we have applied \eqref{eq:h-Hess}, \eqref{eq:Dirichlet}, and \eqref{eq:h bar h closeness}. While by \eqref{eq:h radial},
\begin{align*}
    \dashint_{B_{\bar g_i}(p_i,2A)}|\nabla_{\bar g_i}h_i |^2\, d\bar g_i &\ \le \dashint_{B_{\bar g_i}(p_i,2A)} (8A+\Psi(\varepsilon_i\,|\,A))\, d\bar g_i\le 9nA,
    \\
     \dashint_{B_{\bar g_i}(p_i,2A)}|\nabla_{\bar g_i}\bar h_i |^2\, d\bar g_i &\ \le 2\dashint_{B_{\bar g_i}(p_i,2A)}|\nabla_{\bar g_i}h_i |^2\, d\bar g_i + 2\dashint_{B_{\bar g_i}(p_i,2A)}|\nabla_{\bar g_i}(h_i-\bar h_i)|^2\, d\bar g_i\le 20nA,
\end{align*}
we get
\begin{eqnarray*}
    && \dashint_{B_{\bar g_i}(p_i,2A)}\big||\nabla_{\bar g_i}\bar h_i|^2-|\nabla_{\bar g_i}h_i|^2\big|\,d\bar g_i
    \\
    &=& \dashint_{B_{\bar g_i}(p_i,2A)}\big|\nabla_{\bar g_i}\bar h_i-\nabla_{\bar g_i}h_i\big|\,\big|\nabla_{\bar g_i}\bar h_i+\nabla_{\bar g_i}h_i\big|\, d\bar g_i\\
    &\leq& \left(\dashint_{B_{\bar g_i}(p_i,2A)}|\nabla_{\bar g_i}(\bar h_i-h_i)|^2\,d\bar g_i\right)^{\frac{1}{2}} \left(\dashint_{B_{\bar g_i}(p_i,2A)}2\big(|\nabla_{\bar g_i}h_i|^2+|\nabla_{\bar g_i}\bar h_i|^2\big)  d\bar g_i\right)^{\frac{1}{2}}\\
    &\leq& \Psi(\varepsilon_i\,|\,A).
\end{eqnarray*}
Combining the above with \eqref{eq:h bar h closeness} and \eqref{eq:h radial} again finishes the proof of item (3) of Definition \eqref{Def_conical function}.
\\

\noindent \textbf{Verification of (4) $\displaystyle \dashint_{B_{\bar g_i}(p_i,A)}|\nabla_{\bar g_i}(\bar h_i-d_{\bar g_i}^2(p_i,\cdot))|\,d\bar g_i\leq \psi(\varepsilon_i\,|\,A)A$.}
\\

Again, we start our argument with $h_i$. It is clear from \eqref{eq:grad-lower} and \eqref{eq:grad h upper bound} that 
\begin{align*}
    \sup_{B_{\bar g_i(p_i,2A)}}\Big|\langle\nabla_{\bar g_i}h_i,\nabla_{\bar g_i} d_{\bar g_i}(p_i,\cdot)\rangle - 
    2d_{\bar g_i}(p_i,\cdot)\Big| \le \Psi(\varepsilon_i\,|\,A).
\end{align*}
Thus, by \eqref{eq:h-close-to-d2}, \eqref{eq:h radial}, and the following fact
\begin{align*}
    \left|\nabla_{\bar g_i}(h_i-d^2_{\bar g_i}(p_i,\cdot))\right|^2= \left|\nabla_{\bar g_i}h_i\right|^2+4d_{\bar g_i}^2(p_i,\cdot)-4d_{\bar g_i}(p_i,\cdot)\langle\nabla_{\bar g_i}h_i,\nabla_{\bar g_i} d_{\bar g_i}(p_i,\cdot)\rangle,
\end{align*}
we have
\begin{align*}
    \sup_{B_{\bar g_i(p_i,2A)}}\big|\nabla_{\bar g_i}( h_i-d^2_{\bar g_i}(p_i,\cdot))\big|\le \Psi(\varepsilon_i\,|\,A),
\end{align*}
and (4) follows immediately from \eqref{eq:h bar h grad L2 close} and the Cauchy-Schwarz inequality.
\\

\noindent\textbf{Verification of (2) $\displaystyle \dashint_{B_{\bar g_i}(p_i,A)}|\nabla^2_{\bar g_i}\bar h_i-2\bar g_i|^2\leq \psi(\varepsilon_i\,|\,A)^2$.}
\\

Because of \eqref{eq:h-Hess}, we need only to show that $\nabla^2_{\bar g_i}\bar h_i$ and $\nabla^2_{\bar g_i} h_i$ are close in the $L^2$-sense. To this end, we recall the standard Bochner formula
\begin{align}\label{eq:Bochner}
    2|\nabla^2_{\bar g_i}(\bar h_i-h_i)|^2=&\ \Delta_{\bar g_i}|\nabla_{\bar g_i}(\bar h_i-h_i)|^2-2\Ric_{\bar g_i}(\nabla_{\bar g_i}(\bar h_i-h_i),\nabla_{\bar g_i}(\bar h_i-h_i))
    \\\nonumber
     &\ -2\langle \nabla_{\bar g_i}(\bar h_i-h_i),\nabla_{\bar g_i}\Delta_{\bar g_i}(\bar h_i-h_i)\rangle.
\end{align}
Let $\varphi:M_i\to[0,1]$ be a cut-off function such that
\begin{align*}
    \varphi(x)=1 \quad & \text{ on } \quad B_{\bar g_i}(p_i,A),
    \\
    \varphi(x)=0\quad & \text{ on } \quad M_i\setminus B_{\bar g_i}(p_i,2A),
    \\
    |\nabla_{\bar g_i}\varphi|\le \frac{C(n)}{A},\quad |\Delta_{\bar g_i}\varphi|\le \frac{C(n)}{A^2}\quad &\text{ on } \quad M_i.
\end{align*}
One may either implement the cut-off function of Cheeger-Colding \cite{CC96}, or construct it explicitly by $\varphi=\phi(h_i/A^2)$, where $\phi$ is a decreasing function with $\phi|_{(-\infty,2]}=1$, $\phi|_{[3,+\infty)}=0$, and $-2\le \phi'\le 0$. It is easy to check such a cut-off function satisfies the above conditions.

Now, multiplying \eqref{eq:Bochner} with $\varphi$ and integrate, we get
\begin{align*}
    \int_{B_{\bar g_i}(p_i,A)}2|\nabla^2_{\bar g_i}(\bar h_i-h_i)|^2\,d\bar g_i \le&\  \int_{M_i} 2\varphi|\nabla^2_{\bar g_i}(\bar h_i-h_i)|^2\,d\bar g_i
    \\
    \le&\  \int_{B_{\bar g_i}(p_i,2A)}2(|\Delta_{\bar g_i}\varphi|+2n\varepsilon_ir_i^2)|\nabla_{\bar g_i}(\bar h_i-h_i)|^2\,d\bar g_i
    \\
    &\ + 2 \int_{B_{\bar g_i}(p_i,2A)} |\nabla_{\bar g_i}\varphi||\nabla_{\bar g_i}(\bar h_i-h_i)||\Delta_{\bar g_i}(\bar h_i-h_i)|\,d\bar g_i
    \\
    &\ + 2 \int_{B_{\bar g_i}(p_i,2A)} \varphi|\Delta_{\bar g_i}(\bar h_i-h_i)|^2\,d\bar g_i
    \\
    \le &\ \Psi(\varepsilon_i\,|\, A) |B_{\bar g_i}(p_i,2A)|, 
\end{align*}
where we have applied \eqref{eq:h bar h grad L2 close}, \eqref{eq:Dirichlet}, \eqref{eq:h-Hess}, and the Cauchy-Schwarz inequality. Since the Ricci curvature condition implies a volume doubling estimate, we get
\begin{align*}
    \dashint_{B_{\bar g_i}(p_i,A)}|\nabla^2_{\bar g_i}(\bar h_i-h_i)|^2\,d\bar g_i \le \Psi(\varepsilon_i\,|\, A).
\end{align*}

By \eqref{eq:condition of blow-up sequence}, we can take a noncollapsed pointed-Gromov-Hausdorff limit for a subsequence of $\{(M_i,d_{\bar g_i},p_i)\}_{i=1}^\infty$, and the limit we shall denote by $(X,d,p)$. By the existence of conical functions $\bar h_i$ and Theorem \ref{Thm:conical function implying almost cone}, $(X,d)$ is a metric cone at $p$. On the other hand, by \eqref{eq:harrad lower bound}, the metric on $B_{\bar g_i}(p_i,\delta)$ is uniformly $C^{1,\alpha}$ equivalent to the Euclidean metric, so on the limit $(X,d)$, $p$ is a $C^{1,\alpha}$-smooth point. Namely, there is a smooth neighborhood at $p$, on which the metric $d$ is induced by a $C^{1,\alpha}$ Riemannian metric. However, a metric cone with a smooth tip is the Euclidean space $\mathbb{R}^n$.

Finally, with condition \eqref{eq:condition of blow-up sequence} and Colding's volume continuity theorem \cite{Cold97}, we have
\begin{align*}
    \omega_n(1-\delta)=\lim_{n\to\infty}|B_{\bar g_i}(p_i,1)|=|B_e(1)|=\omega_n,
\end{align*}
which is a contradiction. This completes the proof of the lemma.
\end{proof}

We now use the lemma to show Theorem \ref{Thm:Ric gap}.

\begin{proof}[Proof of Theorem \ref{Thm:Ric gap}] We again argue by contradiction. Assume that there is a sequence of Ricci shrinkers $\{(M_i^n, g_i,f_i, p_i)\}_{i=1}^\infty$ with 
\begin{align}\label{eq: Ric contradiction assum}
    -\varepsilon_ig_i\le\Ric_{g_i}\leq \varepsilon_ig_i\quad \text{ on }\quad B_{g_i}(p_i,1),
\end{align}
where $\varepsilon_i\to0$, but none of them is isometric to $\mathbb{R}^n$. Let us define
\begin{align*}
    r_i=\sup\{s\in(0,1]\,:\,|B_{g_i}(p_i,s)|\ge \omega_n(1-\delta)s^n\}.
\end{align*}
By Lemma \ref{Lm:critical radius estimate}, we have 
\begin{align}\label{eq:critical radius bound}
    r_i\ge r_0(n)>0\quad \text{for all $i$ large enough.}
\end{align}  
Thus, by Lemma \ref{Lm:volume and entropy noncollapsed}, the sequence is noncollapsed in that
\begin{align*}
    \mu(M_i)\ge C(n)\quad \text{ for all $i$ large enough.}
\end{align*}

We may then apply Theorem \ref{Thm:shrinker limit} to take a limit for a subsequence of $\{(M_i^n, g_i,f_i, p_i)\}_{i=1}^\infty$. Denote by $(X,d,p)$ the limit, and $X=\mathcal{R}\sqcup\mathcal{S}$ its decomposition. By the contradictory assumption \eqref{eq: Ric contradiction assum}, we have
\begin{align*}
    \Ric\equiv 0\quad \text{ on }\quad B(p,1)\cap\mathcal{R},
\end{align*}
so $(X,d)$ is a Ricci-flat cone at $p$.

On the other hand, by \eqref{eq:critical radius bound} and Theorem \ref{Thm:critical radius}, we have
\begin{align*}
    r_H(p_i)\ge r_0(n)\delta >0\quad \text{ for all $i$ large enough.}
\end{align*}
This again implies that $p\in X$ is a smooth point; see the argument towards the end of the proof of Lemma \ref{Lm:critical radius estimate}. Thus $X$ is isometric to $\mathbb{R}^n$. Finally, by \cite[Corollary 7]{LW20}, any shrinker that is close enough to the Euclidean space in the pointed Gromov-Hausdorff sense must be isometric to the Euclidean space. So $(M_i^n, g_i,f_i, p_i)$ must be the gaussian shrinker for all $i$ large enough, and this contradicts our assumptions that none of these is gaussian. The proof of the theorem is now completed.
\end{proof}

\section{The proof of Theorem \ref{Thm:nu gap}}
\begin{proof}[Proof of Theorem \ref{Thm:nu gap}]
    Suppose on the contrary. There is a sequence of non-gaussian Ricci shrinkers $\{(M_i^n,g_i,p_i,f_i)\}_{i=1}^\infty$ with 
    \begin{align}\label{eq:entropy lower to 0}
        0\ge \nu(B_{g_i}(p_i,1),g_i,1) \to 0.
    \end{align}
    By the pseudolocality result \cite[Theorem 25]{LW20}, there are fixed positive constants $\varepsilon(n)$ and $C(n)$ such that
\begin{align}\label{eq:smoothness of the tip}
    |{\Rm}_{g_i}| \leq C \quad \text{  on  } \quad B_{g_i}(p_i, \varepsilon).
\end{align}

The local entropy bound \eqref{eq:entropy lower to 0} and a standard argument in Perelman's proof of the no local collapsing theorem \cite{Per02} show that there is a uniform lower bound of $|B_{g_i}(p_i,1)|$. In other words, the sequence $\{(M_i^n,g_i,p_i,f_i)\}_{i=1}^\infty$ is noncollapsed in that
\begin{align*}
    \mu(M_i)\ge -Y\quad \text{ for all $i$ large enough }
\end{align*}
for some $Y<+\infty$. We can now apply Theorem \ref{Thm:shrinker limit} to obtain a limit for a subsequence of $\{(M_i^n,g_i,p_i,f_i)\}_{i=1}^\infty$. Denote by $(X,d,p)$ the limit, $X=\mathcal{R}\sqcup\mathcal{S}$ its decomposition, and $g$ the metric on $\mathcal{R}$. 

By \eqref{eq:smoothness of the tip}, we have that $p$ must be a smooth point. That is, we can find some radius $r<\varepsilon<1$, such that $B(p,r)\subset \mathcal{R}$. Since the convergence of Theorem \ref{Thm:shrinker limit} is smooth on $\mathcal{R}$, we have
\begin{align*}
    \nu(B(p,r),g,1)\ge \lim_{i\to\infty} \nu(B_{g_i}(p_i,r),g_i,1)\ge \lim_{i\to\infty} \nu(B_{g_i}(p_i,1),g_i,1)=0.
\end{align*}
By a result of Cheng \cite[Theorem 1.4]{Cheng23}, $g$ must be a flat metric on $B(p,r)$. As before, $(X,d)$ must be a cone at $p$ while $p$ is a smooth point. So $X$ is the Euclidean space. When $i$ is large enough, \cite[Corollary 7]{LW20} leads to a contradiction.
\end{proof}

\section{The proof of Theorem \ref{Thm:Applications}}

We prove Theorem \ref{Thm:Applications} in this section. Let $(M^n,g_t)_{t\in[-T,0]}$ be a Type I Ricci flow satisfying 
\begin{align*}
    |{\Rm_{g_t}}|\le \frac{C}{|t|}\quad \text{ on }\quad M\times[-T,0).
\end{align*}
Assume by contradiction that $x\in M$ is a point that satisfies the condition in the statement, with $\delta=\delta(n)>0$ a dimensional constant to be fixed, yet $x$ is a singular point in the sense of \cite{EMT11} (cf. \cite[Definition 1.1, 1.2, 1.3, 1.4, 1.5]{EMT11}, note that all these definitions agree according to \cite[Theorem 1.2]{EMT11}). Then, according to \cite[Theorem 1.1]{EMT11}, for any sequence $\lambda_i\searrow 0$, the blow-up sequence $\{(M,\lambda_i^{-1}g_{\lambda_it},x)_{t\in[-T\lambda_i^{-1},0)}\}_{i=1}^\infty$, after passing to a subsequence, converges smoothly to (the canonical flow of) a normalized shrinker:
\begin{align*}
    (M,\lambda_i^{-1}g_{\lambda_it},x)_{t\in[-T\lambda_i^{-1},0)}\xrightarrow{\ \ \ C^\infty \ \ }(M^\infty,g^\infty_t,x_\infty)_{t\in (-\infty,0)}.
\end{align*}
In other words, the limit flow is genetated by a Ricci shrinker $(M^\infty,g,f,p)$:
\begin{align*}
    \partial_t\phi_t:=|t|^{-1}\nabla_g f\circ\phi_t,\quad \phi_{-1}=\operatorname{id},\quad g^\infty_t:=|t|^{-1}\phi_t^*g,\quad t\in(-\infty,0).
\end{align*}

We shall then estimate the distance between $x_\infty$ and the minimum point $p$ of $f$. To this end, let us recall that, by Naber \cite{Na10} $f\circ\phi_t$ is indeed the $C^{0,\alpha}_{\operatorname{loc}}$ or $W^{1,2}_{\operatorname{loc}}$ limit of the singular reduced distance $\ell_{x,0}(\cdot,\lambda_i|t|)$ (that is, the reduced distance based at a singular point, see \cite[Proposition 3.8]{Na10}). It is easy to observe from \cite[Proposition 3.2]{Na10} that $f\circ\phi_t$, as the limit of $\ell_{x,0}(\cdot,\lambda_i|t|)$, satisfies the normalization equation
\begin{align*}
    \left|\nabla_{g^\infty_t} (f\circ\phi_t)\right|_{g^\infty_t}^2+R_{g^\infty_t}=\frac{f\circ\phi_t}{|t|}.
\end{align*}
Thus, the potential function $f:M^\infty\to\mathbb{R}$ is normalized exactly as \eqref{eq:normalization}.

Recall that $p\in M^\infty$ is defined as a minimum point of $f$. To estimate the distance between $x_\infty$ and $p$, we compute $f(x_\infty)$ via $\ell_{x,0}(x,|t|)$. By our assumption 
\begin{align}\label{eq:assumption of rc}
    r_c(x,t)\ge 6n\sqrt{|t|} \quad \text{ for all } \quad  t\in  [-T,0),
\end{align} 
where 
        \begin{align*}
            r_c(x,t):=\sup\{r>0\,:\, \Ric_{g_t}\le \delta r^{-2}g_t\ \text{ on }\ B_{g_t}(x,r)\},
        \end{align*}
         we have
        \begin{align*}
            R_{g_t}(x)\le \frac{\delta}{|t|}\quad \text{ for all } \quad t\in [-T,0).
        \end{align*}
        Let us fix any $t\in [-T,0)$, then, for $t\ll t'<0$, we have
        \begin{align*}
            \ell_{x,t'}(x,|t-t'|)\le &\ \frac{1}{2\sqrt{|t-t'|}}\int_{0}^{|t-t'|}\sqrt{\tau}R_{g_{t'-\tau}}(x)\,d\tau
            \le  \frac{1}{2\sqrt{|t-t'|}}\int_{0}^{|t-t'|}\sqrt{\tau}\frac{\delta}{\tau-t'}\,d\tau
            \\
            \le &\ \frac{1}{2\sqrt{|t-t'|}}\int_0^{|t-t'|}\delta \tau^{-1/2}\,d\tau=\delta.
        \end{align*}
        According to the definition of $\ell_{x,0}(x,|t|)$, we have
        \begin{align*}
            \ell_{x,0}(x,|t|):=\lim_{t'\nearrow 0} \ell_{x,t'}(x,|t|)\le \delta,
        \end{align*}
        and hence 
        \begin{align}\label{eq:center estimate 1}
            f(x_\infty)=\lim_{i\to\infty}\ell_{x,0}(x,\lambda_i)\le \delta.
        \end{align}
        Recall the potential function estimate of \cite[Lemma 2.1]{HM11}   \begin{align}\label{eq:potential estimate}
            \frac{1}{4}\big(d(x,p)-5n\big)_+^2\le f(x)\le \frac{1}{4}\left(d(x,p)+\sqrt{2n}\right)^2.
        \end{align}
        Combining \eqref{eq:center estimate 1} and \eqref{eq:potential estimate}, we have
        \begin{align}\label{eq:position of the center}
            d(x_\infty,p)\le 5n+16\delta^2\le 5n+1. 
        \end{align}

        By our assumption that $r_{c}(x,t)\ge 6n\sqrt{|t|}$, we immediately get 
        \begin{align*}
            \Ric_g\le \delta g\quad \text{ on }\quad B(x_\infty,6n).
        \end{align*}
        Combining  \eqref{eq:position of the center}, we have
        \begin{align*}
            \Ric_g\le \delta g\quad \text{ on }\quad B(p,1).
        \end{align*}
        It follows from Theorem \ref{Thm:Ric gap} that the shrinker $(M^\infty,g,f,p)$ is gaussian, if we take $\delta\le \varepsilon$, where $\varepsilon$ is the constant therein. Finally, by \cite[Theorem 1.4]{EMT11}, $x$ is not a singular point.

\bigskip


\begin{thebibliography}{CCG{\alphalchar{+}}10}

\bibitem[And90]{And90}
Anderson, Michael T. \emph{Convergence and rigidity of manifolds under Ricci curvature bounds}.
Invent. Math. 102 (1990), no. 2, 429–445.


\bibitem[CZ10]{CZ10}Cao, Huai-Dong; Zhou, De-Tang. \emph{On complete gradient shrinking Ricci solitons.} J. Differential Geom. 85 (2010), 175--186. 
\bibitem[CN09]{CN09} Carrillo, Jos\'e A.; Ni, Lei \emph{Sharp logarithmic Sobolev inequalities on gradient solitons and applications}. Comm. Anal. Geom. 17 (2009), no. 4, 721–753.



\bibitem[CMZ25]{CMZ22} Chan, Pak-Yeung; Ma, Zilu; Zhang, Yongjia. \emph{A local gap theorem for Ricci shrinkers.} Comm. Anal. Geom. 2025, to appear.


\bibitem[CJN21]{CJN21}
Cheeger, Jeff; Jiang, Wenshuai; Naber, Aaron.  \emph{Rectifiability of singular sets of noncollapsed limit spaces with Ricci curvature bounded below}. Ann. of Math. 193(2021), no. 2, 407–538.
 



\bibitem[CC96]{CC96} Cheeger, Jeff; Colding, Tobias H. \emph{Lower Bounds on Ricci Curvature and the Almost Rigidity of Warped Products} Annals of Mathematics, 144.1(1996),189-237 

\bibitem[CC97]{CC97} Cheeger, Jeff; Colding, Tobias H. \emph{On the structure of spaces with Ricci curvature bounded below. I.} Journal of Differential Geometry 46.3 (1997): 406-480.

\bibitem[CC00a]{CC00a} Cheeger, Jeff; Colding, Tobias H. \emph{On the structure of spaces with Ricci curvature bounded below. II.} Journal of Differential Geometry 54.1 (2000): 13-35.

\bibitem[CC00b]{CC00b} Cheeger, Jeff; Colding, Tobias H. \emph{On the structure of spaces with Ricci curvature bounded below. III.} Journal of Differential Geometry 54.1 (2000): 37-74.


\bibitem[Ch01]{Ch01} Cheeger, Jeff. \emph{Degeneration of Riemannian metrics under Ricci curvature bounds.}  Scuola Normale Superiore, 2001.


\bibitem[Che09]{Che09}Chen, Bing-Long, \emph{Strong uniqueness of the Ricci flow}. J. Differential Geom. 82 (2009), 363--382.


\bibitem[Cheng23]{Cheng23} Cheng, Liang. \emph{On local rigidity theorems with respect to the scalar curvature}. 	 arXiv preprint arXiv:2310.05011.


\bibitem[Co97]{Cold97}Colding, Tobias H. \emph{Ricci curvature and volume convergence}. Ann. of Math. (2) 145 (1997), no. 3, 477–501.

\bibitem[CM25]{CM25} Colding, Tobias Holck;  Minicozzi, William P. \emph{Singularities of Ricci flow and diffeomorphisms.} Publications mathématiques de l'IHÉS 142.1 (2025): 75-152.









\bibitem[EMT11]{EMT11} Enders, Joerg; M\"uller, Reto; Topping, Peter M. \emph{On type-I singularities in Ricci flow}. Comm. Anal. Geom. 19 (2011), no. 5, 905–922.





\bibitem[HM11]{HM11} Haslhofer, Robert; M\"uller, Reto. \emph{A compactness theorem for complete Ricci shrinkers}, Geom.
Funct. Anal. 21 (2011), no. 5, 1091–1116, DOI 10.1007/s00039-011-0137-4. MR2846384






\bibitem[HLW21]{HLW21} Huang, Shaosai; Li, Yu; Wang, Bing  \emph{On the regular-convexity of Ricci shrinker limit spaces}. J. Reine Angew. Math. 771 (2021), 99–136.

\bibitem[H19]{H19} Huang, Shaochuang  \emph{ A note on existence of exhaustion functions and its applications}. J. Geom. Anal. 29 (2019), no. 2, 1649–1659.

\bibitem[J]{J} Jiang, Wenshuai. \emph{Lectures on Cheeger-Colding Theory}, unpublished.



\bibitem[LW20]{LW20} Li, Yu; Wang, Bing. \emph{Heat kernel on Ricci shrinkers}. Calc. Var. Partial Differential Equations 59 (2020), no. 6, Paper No. 194, 84 pp.
\bibitem[LLW21]{LLW21}Li, Haozhao; Li, Yu; Wang, Bing. \emph{On the structure of Ricci shrinkers} J. Funct. Anal. 280 (2021), no. 9, Paper No. 108955, 75 pp.



\bibitem[LW24]{LW24} Li, Yu; Wang, Bing. \emph{Rigidity of the round cylinders in Ricci shrinkers.} Journal of Differential Geometry 127.2 (2024): 817-897.


\bibitem[LZ23]{LZ23} Li, Yu; Zhang, Wenjia. \emph{On the rigidity of Ricci shrinkers.} arXiv preprint arXiv:2305.06143 (2023).

\bibitem[MW11]{MW11} Munteanu, Ovidiu; Wang, Mu-Tao. \emph{The curvature of gradient Ricci solitons}. Math. Res. Lett. 18 (2011), no. 6, 1051–1069.

\bibitem[Na10]{Na10} Naber, Aaron. \emph{Noncompact shrinking four solitons with nonnegative curvature.} Journal für die reine und angewandte Mathematik, vol. 2010, no. 645, 2010, pp. 125-153. https://doi.org/10.1515/crelle.2010.062



\bibitem[Per02]{Per02} Perelman, Grisha, \emph{The entropy formula for the Ricci flow and its geometric applications}, arXiv:math.DG/0211159 (2002).




\bibitem[WW25]{WW25} Wang, Jie; Wang, Youde. \emph{Rigidity and $\varepsilon$-regularity theorems of Ricci shrinkers} Calc. Var. Partial Differential Equations 64 (2025), no. 2, Paper No. 42, 27 pp.



\bibitem[Yo12]{Yo12} Yokota, Takumi. \emph{Addendum to ‘Perelman’s reduced volume and a gap theorem for the Ricci flow’.} Comm. Anal. Geom. 20 (2012): 949-955.


\bibitem[Zh18]{Zh18} Zhang, Shijin. \emph{A gap theorem on complete shrinking gradient Ricci solitons}. Proc. Amer. Math. Soc. 146 (2018), no. 1, 359–368.

\bibitem[Zh20]{Zh20} Zhang, Zhuhong. \emph{A gap theorem of four-dimensional gradient shrinking solitons}. Comm. Anal. Geom. 28 (2020), no. 3, 729–742
\end{thebibliography}
\bibliographystyle{amsalpha}

\newcommand{\alphalchar}[1]{$^{#1}$}
\providecommand{\bysame}{\leavevmode\hbox to3em{\hrulefill}\thinspace}
\providecommand{\MR}{\relax\ifhmode\unskip\space\fi MR }
\providecommand{\MRhref}[2]{%
  \href{http://www.ams.org/mathscinet-getitem?mr=#1}{#2}
}

\noindent Department of Mathematics, National Tsing Hua University, Hsinchu, Taiwan
\\ E-mail address: \verb"pychan@math.nthu.edu.tw"
\\

\noindent School of Mathematical Sciences, Shanghai Jiao Tong University, Shanghai, 200240, China
\\ E-mail address: \verb"sunzhang91@sjtu.edu.cn"

\end{document}